\documentclass[11pt]{article}
\usepackage{amsmath}
\usepackage{amsthm}
\usepackage{amssymb}
\title{
Holonomic rank of ${\mathcal A}$-hypergeometric
differential-difference equations}
\author{
Katsuyoshi Ohara\footnote{Department of Computational Science,
Kanazawa University} and
Nobuki Takayama\footnote{Department of Mathematics, Kobe University}
}
\date{June 18, 2007}
\newtheorem{theorem}{Theorem}

\newtheorem{definition}{Definition}
\newtheorem{example}{Example}

\newtheorem{lemma}{Lemma}

\newcommand{\Ker}{\operatorname{Ker}}
\newcommand{\rank}{\operatorname{rank}}
\newcommand{\init}{\operatorname{in}}

\newcommand{\vol}{\operatorname{vol}}

\def\HH{{\boldsymbol{H}}}
\def\pd#1{\partial_{#1}}
\def\A{{\mathcal A}}
\def\C{{\mathbf C}}
\def\D{{\boldsymbol{D}}}

\def\R{{\mathbf R}}
\def\U{{\mathbf U}}
\def\Q{{\mathbf Q}}
\def\Z{{\mathbf Z}}
\numberwithin{equation}{section}

\begin{document}

\maketitle

\begin{abstract}
We introduce $\A$-hypergeometric differential-difference
equation $\HH_A$ and prove that its holonomic rank is equal 
to the normalized volume of $\A$
with giving a set of convergent series solutions.
\end{abstract}

\section{Introduction}

In this paper, we introduce $\A$-hypergeometric differential-difference
equation $\HH_A$ and study its series solutions and holonomic rank.

Let $A=(a_{ij})_{i=1,\ldots,d,j=1,\ldots,n}$ be a $d\times n$-matrix whose elements are integers.
We suppose that the set of the column vectors of $A$ spans ${\bf Z}^d$
and there is no zero column vector.
Let $a_i$ be the $i$-th column vector of the matrix $A$ and $F(\beta,x)$
the integral
\[
F(\beta,x) = \int_C \exp\left(\sum_{i=1}^n x_i t^{a_i} \right)
t^{-\beta-1} dt, \qquad t=(t_1, \ldots, t_d), \ 
\beta = (\beta_1, \ldots, \beta_d).
\]
The integral $F(\beta,x)$ satisfies the $\A$-hypergeometric
differential system associated to $A$ and $\beta$ ``formally''.
We use the word ``formally''
because, there is no general and rigorous description about the cycle $C$
(\cite[p.222]{SST}).

We will regard the parameters $\beta$ as variables.
Then, the function  $F(s,x)$ on the $(s,x)$ space satisfies
differential-difference equations ``formally'',
which will be our ${\cal A}$-hypergeometric differential-difference system.

Rank theories of $\A$-hypergeometric differential system have been
developed since Gel'fand, Zelevinsky and Kapranov \cite{GZK}.
In the end of 1980's, under the condition that the points lie on a same
hyperplane,
they proved that the rank of
$\A$-hypergeometric differential system $H_A(\beta)$ agrees with
the normalized volume of $A$ for any parameter $\beta \in {\bf C}^d$ 
if the toric ideal $I_A$ has the
Cohen-Macaulay property.  After their result had been gotten, 
many people have studied on conditions such that the rank equals the normalized volume.
In particular, Matusevich, Miller and Walther proved that $I_A$ has the
Cohen-Macaulay property if the rank of $H_A(\beta)$ agrees with the
normalized volume of $A$ for any $\beta \in {\bf C}^d$ (\cite{MMW}).


In this paper, we will introduce ${\cal A}$-hypergeometric differential-difference
system, which can be regarded as a generalization of difference equation
for the $\Gamma$-function, the Beta function, and the Gauss hypergeometric difference
equations.
As the first step on this differential-difference system, 
we will prove our main Theorem~\ref{th:main} utilizing theorems on ${\cal A}$-hypergeometric
differential equations,
construction of convergent series solutions with a homogenization
technique, 
uniform convergence of series solutions, and Mutsumi Saito's results for contiguity
relations \cite{Saito1}, \cite{SST-compositio}, \cite[Chapter 4]{SST}.
The existence theorem \ref{th:converge} on convergent series fundamental set of solutions 
for ${\cal A}$-hypergeometric differential equation for generic $\beta$
is the second main theorem of our paper.
Finally, we note that,
for studying our $\A$-hypergeometric
differential-difference system, we wrote a program
``yang''~(\cite{Oh}, \cite{OpenXM})
on a computer algebra system Risa/Asir and did several experiments
on computers to conjecture and prove our theorems.

\section{Holonomic rank}
Let $\D$ be the ring of differential-difference operators
\[
\C
\langle
x_1, \ldots, x_n, s_1, \ldots, s_d,
\pd{1}, \ldots, \pd{n}, S_1, \ldots, S_d, S_1^{-1}, \ldots,
S_d^{-1} \rangle
\]
where the following (non-commutative) product rules are assumed
\[
S_i s_i = (s_i+1) S_i,
\quad
S_i^{-1} s_i = (s_i-1) S_i^{-1},
\quad
\pd{i} x_i = x_i \pd{i} + 1
\]
and the other types of the product of two generators commute.

Holonomic rank of a system of differential-difference equations
will be defined by using the following ring of differential-difference
operators with rational function coefficients
\[
\U = \C(s_1, \ldots, s_d, x_1, \ldots, x_n)
\langle S_1, \ldots, S_d, S_1^{-1}, \ldots, S_d^{-1},
\pd{1}, \ldots, \pd{n} \rangle
\]
It is a $\C$-algebra generated by rational functions in
$s_1, \ldots, s_d, x_1, \ldots, x_n$
and differential operators
$ \pd{1}, \ldots, \pd{n} $
and difference operators
$S_1, \ldots, S_d, S_1^{-1}, \ldots, S_d^{-1}$.
The commutation relations are defined by
$  \pd{i} c(s,x) = c(s,x) \pd{i} + \frac{\partial c}{\partial x_i} $,
$  S_{i} c(s,x) = c(s_1,\ldots, s_i+1, \ldots, s_d,x) S_{i}$,
$  S_{i}^{-1} c(s,x) = c(s_1,\ldots, s_i-1, \ldots, s_d,x) S_{i}^{-1}$.

Let $I$ be a left ideal in $\D$.
The holonomic rank of $I$ is the number
\[
\rank(I) = {\rm dim}_{\C(s,x)} \U/(\U I).
\]
In case of the ring of differential operators ($d=0$),
the definition of the holonomic rank agrees with the standard definition
of holonomic rank in the ring of differential operators.

For a given left ideal $I$, the holonomic rank can be evaluated
by a Gr\"obner basis computation in $\U$.


\section{$\A$-hypergeometric differential-difference equations}

Let $A=(a_{ij})_{i=1,\ldots,d,j=1,\ldots,n}$ be an integer $d \times n$ matrix of rank $d$.
We assume that the column vectors $\{ a_i \}$ 
of $A$ generates $\Z^d$ and there is no zero vector.
The {\it $\A$-hypergeometric differential-difference system}
$\HH_A$ is the following system of differential-difference equations
\begin{eqnarray*}
  \left( \sum_{j=1}^n a_{ij} x_j \pd{j} - s_i \right) \bullet f &=& 0
   \qquad\mbox{ for } i = 1, \ldots, d  \quad\mbox{ and } \\
  \left( \pd{j} - \prod_{i=1}^n S_{i}^{-a_{ij}} \right) \bullet f &= & 0
   \qquad\mbox{ for } j = 1, \ldots, n.
\end{eqnarray*}
Note that $\HH_A$ contains the toric ideal $I_A$.
(use~\cite[Algorithm 4.5]{S} to prove it.) 

\begin{definition} \rm
Define the unit volume in $\R^d$ as the volume of
the unit simplex $\{ 0, e_1, \ldots, e_d \}$.
For a given set of points $\A=\{ a_1 , \ldots, a_n \}$ in 
$\R^d$,
the normalized volume $\vol(\A)$ is 
the volume of the convex hull of the origin and $\A$.
\end{definition}

\begin{theorem}\label{thm:solutions}
$\A$-hypergeometric differential-difference system $\HH_A$ has
linearly independent
$\vol(A)$ series solutions.
\end{theorem}

The proof of this theorem is divided into two parts.
The matrix $A$ is called homogeneous when it contains a row
of the form $(1,\ldots, 1)$.
If $A$ is homogeneous, then the associated toric ideal $I_A$ 
is homogeneous ideal \cite{S}.
The first part is the case that $A$ is homogeneous.
The second part is the case that $A$ is not homogeneous.

\begin{proof} ($A$ is homogeneous.)
We will prove the theorem with the homogeneity assumption of $A$.
In other words, we suppose that $A$ is written as follows:
\[
A =
\left(
\begin{matrix}
1 & \cdots & 1 \\
  &  *     &
\end{matrix}
\right).
\]
Gel'fand, Kapranov, Zelevinski gave a method to construct
$m=\vol(A)$ linearly independent solutions of 
$H_A(\beta)$ with the homogeneity condition of $A$
(\cite{GZK}).
They suppose that $\beta$ is fixed as a generic ${\bf C}$-vector.
Let us denote their series solutions by
$f_1(\beta;x), \ldots, f_m(\beta;x)$.
It is easy to see that the functions $f_i(s;x)$ are solutions of 
the differential-difference equations $\HH_A$.
We can show, by carefully checking the estimates of
their convergence proof,
that there exists an open set
in the $(s, x)$ space such that
$f_i(s; x)$ is locally uniformly convergent with respect to $s$ and $x$.
Let us sketch their proof to see that their series
converge as solutions of $\HH_A$.
The discussion is given in~\cite{GZK}, but we need to rediscuss it
in a suitable form to apply it to the case of inhomogeneous $A$. 

Let $B$ be a matrix of which the set of column vectors is a basis of
$\Ker (A: \Q^n \to \Q^d)$ and is normalized as follows:
\[
B = \left(
\begin{array}{ccc}
1 & & \\
  & \ddots & \\
  & & 1 \\
  & * &
\end{array}
\right) \in M(n,n-d,\Q).
\]
We denote by $b^{(i)}$ the $i$-th column vector of $B$ and
by $b_{ij}$ the $j$-th element of $b^{(i)}$.
Then the homogeneity of $A$ implies
\[
\sum_{j=1}^n b_{ij} = 0.
\]

Let us fix a regular triangulation $\Delta$ of 
$\A = \{ a_1, \ldots, a_n\}$ 
following the construction by Gel'fand, Kapranov, Zelevinsky.   
Take a $d$-simplex $\tau$ in the triangulation $\Delta$.
If $\lambda \in \C^n$ is admissible for a $d$-simplex $\tau$ of
$\{1,2,\ldots, n\}$ 
({\it admissible} $\Leftrightarrow$ for all $j \not\in \tau$, $\lambda_j \in \Z$),
and
$A\lambda = s$ holds, then $\HH_A$ has a formal series solution
\[
\phi_\tau(\lambda;x) = \sum_{l\in L} \frac{x^{\lambda+l}}{\Gamma(\lambda+l+1)},
\]
where $L = \Ker(A: \Z^n \to \Z^d)$ and
$\Gamma(\lambda+l+1) = \prod_{i=1}^n \Gamma(\lambda_i+l_i+1)$
and when a factor of the denominator of a term in the sum,
we regard the term is zero.
Put $\# \tau = n'$.
Note that there exists an open set $U$ in the $s$ space
such that $\lambda_i,\  i \in \tau$ lie in a compact set in
${\bf C}^{n'} \setminus {\bf Z}^{n'}$. 
Moreover, this open set $U$ can be taken as a common open set for all 
$d$-simplices in the triangulation $\Delta$ and the associated 
admissible $\lambda$'s 
when the integral values $\lambda_j \ (j \not\in \tau)$ are
fixed for all $\tau \in \Delta$.

Put
$L' = \{(k_1, \ldots, k_{n-d})\in \Z^{n-d} \mid \sum_{i=1}^{n-d} k_i b^{(i)} \in \Z^n \}$.
Then, $L'$ is  $\Z$-submodule of $\Z^{n-d}$ and
$L=\{\sum_{i=1}^{n-d}k_i b^{(i)} \mid k \in L'\}$.
In other words, $L$ can be parametrized with $L'$.
Without loss of the generality, we may suppose that 
$\tau = \{n-d+1,\ldots,n\}$.
Then, we have
\[
\phi_\tau(\lambda;x) =
\sum_{l\in L} \frac{x^{\lambda + l}}{\Gamma(\lambda+l+1)}
=
\sum_{k\in L'} \frac{x^{\lambda + \sum_{i=1}^{n-d}k_i b^{(i)}}}{\Gamma(\lambda+\sum_{i=1}^{n-d}k_i b^{(i)}+1)}
\]
Note that the first $n-d$ rows of $B$ are normalized. 
Then, we have
\[
\lambda_j + \sum_{i=1}^{n-d} k_i b_{ij} + 1 = \lambda_j + k_j + 1 \in \Z
\qquad (j=1, \ldots, n-d)
\]
Since $1/\Gamma(0) = 1/\Gamma(-1) = 1/\Gamma(-2)= \cdots = 0$, 
the sum can be written as
\[
\phi_\tau(\lambda;x)
=
\sum_{\substack{
k\in L'\\
\lambda_j+k_j+1\in \Z_{> 0}\\
(j=1,\ldots,n-d)
}}
\frac{x^{\lambda + \sum_{i=1}^{n-d}k_i b^{(i)}}}{\Gamma(\lambda+\sum_{i=1}^{n-d}k_i b^{(i)}+1)}
\]
Moreover, when we put
\begin{eqnarray*}
k_j'          &=& \lambda_j + k_j, \qquad (j=1, \ldots, n-d) \\
\lambda'      &=& \lambda - \sum_{i=1}^{n-d} \lambda_i b^{(i)} \\
\hat{\lambda} &=& (\lambda_1, \ldots, \lambda_{n-d})
\end{eqnarray*}
we have
\[
\sum_{i=1}^{n-d} k_i b^{(i)} =
- \sum_{i=1}^{n-d} \lambda_i b^{(i)} + \sum_{i=1}^{n-d} k_i'
 b^{(i)}
\]
Hence, the sum $\phi_\tau(\lambda; x)$ can be written as
\begin{eqnarray*}
\phi_\tau(\lambda;x)
&=&
\sum_{\substack{
k'\in L' + \hat{\lambda} \\
k'\in \Z_{\ge 0}^{n-d}
}}
\frac{x^{\lambda - \sum_{i=1}^{n-d}\lambda_i b^{(i)}}\cdot x^{\sum_{i=1}^{n-d}k_i' b^{(i)}}}
{\Gamma(\lambda-\sum_{i=1}^{n-d}\lambda_i b^{(i)}+\sum_{i=1}^{n-d}k_i'b^{(i)}+1)}
\\
&=&
x^{\lambda'}
\sum_{\substack{
k'\in L' + \hat{\lambda} \\
k'\in \Z_{\ge 0}^{n-d}
}}
\frac{(x^{b^{(1)}})^{k_1'}\cdots(x^{b^{(n-d)}})^{k_{n-d}'}}
{\Gamma(\lambda'+\sum_{i=1}^{n-d}k_i' b^{(i)}+1)}
\end{eqnarray*}
Note that our series with the coefficients in terms of Gamma functions
agree with those in \cite[\S 3.4]{SST}, which do not contain Gamma functions,
by multiplying suitable constants.
Hence we will apply some results on series solutions in \cite{SST} 
to our discussions in the sequel.

\begin{lemma}\label{prop:2}
Let $(k_i)\in (\Z_{\ge 0})^m$ and $(b_{ij}) \in M(m,n,\Q)$.
Suppose that 
\[
\sum_{i=1}^m k_i b_{ij} \in \Z,
\qquad
\sum_{j=1}^n b_{ij} = 0
\]
and
parameters $\lambda = (\lambda_1, \ldots, \lambda_n)$ belongs to 
a compact set $K$. 
Then 
there exists a positive number $r$, which is independent of $\lambda$, 
such that the power series 
\[
\sum_{\substack{
k'\in L' + \hat{\lambda} \\
k'\in \Z_{\ge 0}^{n-d}
}}
\frac{(x^{b^{(1)}})^{k_1'}\cdots(x^{b^{(n-d)}})^{k_{n-d}'}}
{\Gamma(\lambda'+\sum_{i=1}^{n-d}k_i' b^{(i)}+1)}
\]
is convergent in $|x^{b^{(1)}}|, \cdots, |x^{b^{(n-d)}}| < r$.
\end{lemma}

The proof of this lemma can be done by elementary estimates of $\Gamma$ functions.
See \cite[pp.18--21]{kaken} if readers are interested in the details.
Since
\[
k' \in L' + \hat{\lambda}
\Longleftrightarrow
\sum_{i=1}^{n-d} k_i' b^{(i)} \in \Z^n
\]
it follows from Lemma~\ref{prop:2} that there exists a positive constant 
$r$ such that the series
converge in 
\begin{equation}\label{eq:radius2}
|x^{b^{(1)}}|, \cdots, |x^{b^{(n-d)}}| < r
\end{equation}
for any $s$ in the open set $U$.
We may suppose $r<1$.
Take the $\log$ of $(\ref{eq:radius2})$.
Then we have
\begin{equation}\label{eq:radius3}
b^{(k)}\cdot (\log |x_1|, \ldots, \log |x_n|) < \log |r| < 0
\quad \forall k \in \{ 1, \ldots, n-d \}
\end{equation}

Following \cite{GZK}, for the simplex $\tau$ and $r$,
we define the set $C(A,\tau,r)$ as follows.
\[
C(A,\tau,r) =
\left\{
\psi \in \R^n  \ \left| \ \exists \varphi \in \R^d,  \ \
\psi_i - (\varphi,a_i)
\begin{cases}
> -\log |r|, & i\not\in \tau, \\
= 0,         & i\in \tau,
\end{cases}
\right.
\right\}
\]
The condition ~$(\ref{eq:radius3})$ and
$(-\log |x_1|, \ldots, -\log |x_n|) \in C(A,\tau,r)$
is equivalent
(see \cite[section 4]{BFS} as to the proof).

Since $\Delta$ is a regular triangulation of $A$,
$
\bigcap_{\tau\in \Delta} C(A,\tau,r)
$
is an open set.
Therefore, when $s$ lies in the open set $U$ and $- \log |x|$ lies in the above
open set, the $\vol(A)$ linearly independent solutions converge. 
\end{proof}

Let us proceed on the proof for the inhomogeneous case.
We suppose that $A$ is not homogeneous and has only non-zero
column vectors.
We define the homogenized matrix as
\[
\tilde{A} =
\left(
\begin{matrix}
1      & \cdots & 1 & 1 \\
a_{11} & \cdots & a_{1n} & 0 \\
\vdots &        & \vdots & \vdots \\
a_{d1} & \cdots & a_{dn} & 0
\end{matrix}
\right)
\in M(d+1,n+1,\Z).
\]
For $s=(s_1, \ldots, s_n) \in \C^d$ and a generic complex number $s_0$,
we put $\tilde s = (s_0, s_1, \ldots, s_d)$.
We suppose that $\tau = \{ n-d+1, \ldots, d, d+1 \}$
is a $(d+1)$-simplex. 
Let us take an admissible $\lambda$ for $\tau$ such that
$\tilde{A} \tilde{\lambda} = \tilde{s}$
and
$\tilde{\lambda} = (\lambda_1, \ldots, \lambda_{n+1})\in \R^{n+1}$
as in the proof of the homogeneous case.
Put $\lambda=(\lambda_1, \ldots, \lambda_n)$.
Consider the solution of the hypergeometric system for $\tilde A$
\[
\tilde{\phi}_\tau(\tilde{\lambda}; \tilde{x}) =
\sum_{k'\in L'\cap S} \frac{
{\tilde {x}}^{\lambda+\sum_{i=1}^{n-d} k_i'b^{(i)}}}
{\Gamma(\lambda+\sum_{i=1}^{n-d} k_i'b^{(i)}+1)}
\]
and the series
\[
\phi_\tau(\lambda; x) =
\sum_{k'\in L'\cap S} \frac{
\prod_{j=1}^n x_j^{\lambda+\sum_{i=1}^{n-d} k_i'b_{ij}}}
{\prod_{j=1}^n \Gamma(\lambda_j+\sum_{i=1}^{n-d} k_i'b_{ij}+1)}
\]
($\tilde{x} = (x_1, \ldots, x_{n+1})$, $x=(x_1, \ldots, x_n)$).
Here, the set $S$ is a subset of $L'$ such that
an integer in $\Z_{\leq 0}$ does  not appear in the arguments
of the Gamma functions in the denominator.
We note that $L'$ for $\tilde{A}$ and $L'$ for $A$ agree,
which can be proved as follows.
Let $(k_1, \ldots, k_{n+1})$ be in the kernel of $\tilde{A}$ in $\Q^{n+1}$. 
Since $\tilde{A}$ contains the row of the form $(1, \ldots, 1)$,
then $(k_1, \ldots, k_n) \in \Z^n$ implies that $k_{n+1}$ is an integer.
The conclusion follows from the definition of $L'$.

\begin{definition} \rm
We call 
$\phi_\tau(\lambda; x)$ the {\it dehomogenization}
of  $\tilde{\phi}_\tau(\tilde{\lambda}; \tilde{x})$.
\end{definition}

Intuitively speaking, the dehomogenization is defined
by ``forgetting'' the last variable $x_{n+1}$ associated $\Gamma$ factors. 
See Example \ref{ex:series}.

Formal series solutions for the hypergeometric system
for inhomogeneous $A$ do not converge in general.
However, we can construct 
$\vol(A)$ convergent series solutions
as the dehomogenization of a set of series solutions
for $\tilde A$ hypergeometric system associated to
a regular triangulation on $\tilde \A $ induced by 
a ``nice'' weight vector $\tilde w (\varepsilon)$,
which we will define.
Put $\tilde{w} = (1,\ldots, 1, 0) \in \R^{n+1}$.
Since the Gr\"obner fan for the toric variety
$I_{\tilde{A}}$ is a polyhedral fan,
the following fact holds. 
\begin{lemma} \label{lemma:interior}
For any  $\varepsilon > 0$,
there exists $\tilde{v}\in \R^{n+1}$ such that
$\tilde{w}(\varepsilon) := \tilde{w} + \varepsilon \tilde{v}$ 
lies in the interior of a maximal dimensional Gr\"obner cone
of $I_{\tilde A}$.
We may also suppose $\tilde{v}_{n+1} =0$.
\end{lemma}

\begin{proof}
Let us prove the lemma.
The first part is a consequence of an elementary property of
the fan. 
When $I$ is a homogeneous ideal in the ring of polynomials of $n+1$ variables, 
we have
\begin{equation}\label{eq:claim1}
\init_{\tilde{u}}(I) = \init_{\tilde{u}+t(1,\cdots,1)}(I)
\end{equation}
for any $t$ and any weight vector $\tilde{u}$.
In other words, 
$\tilde u$ and $\tilde u + t (1, \ldots, 1)$ lie
in the interior of the same Gr\"obner cone.

When the weight vector 
$\tilde w(\varepsilon)=\tilde w + \varepsilon \tilde v$ lies
in the interior of the Gr\"obner cone,
we define a new $\tilde v $ by
$\tilde v - \tilde{v}_{n+1} (1, \ldots, 1)$.
Since the initial ideal does not change with this change of weight,
we may assume that $\tilde{v}_{n+1} = 0$ for the new $\tilde v$.
\end{proof}

Since the Gr\"obner fan is a refinement of the secondary fan
and hence
$\tilde{w}(\varepsilon)$ is an interior point
of a maximal dimensional secondary cone,
it induces a regular triangulation
(\cite{S} p.71, Proposition 8.15).
We denote by $\Delta$ 
the regular triangulation on $\tilde{A}$
induced by $\tilde{w}(\varepsilon)$.
For a $d$-simplex $\tau \in \Delta$,
we define $b^{(i)}$ as in the proof of the homogeneous case.
Since the weight for $\tilde{a}_{n+1}$ is the lowest,
${n+1}\in \tau$ holds.
We can change indices of 
$\tilde{a}_1, \ldots, \tilde{a}_n$ so that
$\tau = \{n-d+1, \ldots, {n+1}\}$ without loss of generality.

Let us prove that the dehomogenized series
$\phi_\tau(\lambda;x)$ converge.
It follows from a characterization of the support of the series
\cite[Theorem 3.4.2]{SST}  that we have
\[
\tilde{w}(\varepsilon) \cdot
\left(\sum_{i=1}^{n-d} k_i' b^{(i)} +\lambda \right)
\ge \tilde{w}(\varepsilon) \cdot \lambda,
\qquad \forall k'\in L'\cap S.
\]
Here, $S$ is a set such that $\Z_{\le 0}$ does not appear in the denominator
of the $\Gamma$ factors.
Take the limit $\varepsilon \to 0$ and we have 
\[
\tilde{w}\cdot \sum_{i=1}^{n-d} k_i' b^{(i)} \ge 0,
\qquad \forall k'\in L'\cap S.
\]
From Lemma \ref{lemma:interior}, 
$\tilde{w}(\varepsilon) \in C(\tilde{A},\tau)$ holds
and then
\[
\tilde{w}(\varepsilon)\cdot b^{(i)} \ge 0.
\]
Similarly, by taking the limit $\varepsilon \to 0$,
we have
\[
\tilde{w}\cdot b^{(i)} = \sum_{j=1}^n b_{ij}\ge 0.
\]
Therefore, we have $\sum_{j=1}^{n+1} b_{ij} = 0$,
the inequality $b_{i,{n+1}}\le 0 $ 
holds for all $i$.

Since
$k'_1 \ge -\lambda_1, \ldots, k_{n-d}'\ge -\lambda_{n-d}$,
we have
\[
\sum_{i=1}^{n-d} k_i' b_{i,{n+1}} \le - \sum_{i=1}^{n-d} \lambda_i b_{i,{n+1}}
\]
Note that the right hand side is a non-negative number.
Suppose that $\lambda_{n+1}$ is negative.
In terms of the Pochhammer symbol we have
$\Gamma(\lambda_{n+1} - m) 
= \Gamma(\lambda_{n+1})(-\lambda_{n+1}+1;m)^{-1}(-1)^m
$, 
then we can estimate the $(n+1)$-th gamma factors as
\begin{eqnarray}
\left|\Gamma(\lambda_{n+1} + \sum_{i=1}^{n-d} k_i'b_{i,{n+1}} +1)\right|
&=& |\Gamma(\lambda_{n+1}+1)| \cdot
\left|\left(-\lambda_{n+1} ; -\sum_{i=1}^{n-d} k_i'b_{i,{n+1}}\right)\right|^{-1}
\nonumber \\
&\le&
c' |\Gamma(\lambda_{n+1}+1)| \cdot
\left|\left(-\lambda_{n+1}; -\sum_{i=1}^{n-d} \lambda_i b_{i,{n+1}}
      \right)\right|^{-1}
\nonumber \\
&=& c \label{eq:eval}
\end{eqnarray}
Here, $c'$ and $c$ are suitable constants.

When $\lambda_{n+1} \geq 0$, there exists only finite set of values  
such that $\lambda_{n+1}+\sum_{i=1}^{n-d} k_i' b_{i,n+1} \geq 0$.
Then, we can show the inequality $(\ref{eq:eval})$ in an analogous way.

Now, by $(\ref{eq:eval})$, we have
\[
\left|
\frac{1}{\prod_{j=1}^n\Gamma(\lambda_j+\sum_{i=1}^{n-d} k_i'b_{ij}+1)}
\right|
\le c
\left|
\frac{1}{\Gamma(\lambda+\sum_{i=1}^{n-d} k_i'b^{(i)}+1)}
\right|
\]
We note that the right hand side is the coefficient
of the series solution for the homogeneous system for
$\tilde A$ and the series converge for
$(-\log|x_1|, \ldots, -\log |x_{n+1}|) \in C(\tilde{A},\tau,r)$
($r<1$)
uniformly with respect to $\tilde s$ in an open set.

Put $x_{n+1}=1$.
Since $-\log |x_{n+1}|=0$
and
$\tilde{w}(\varepsilon) \in \{y \mid y_{n+1} = 0\}$,
we can see that
\[
\bigcap_{\tau\in\Delta} C(\tilde{A},\tau,r) \cap \{y \mid y_{n+1} = 0\}
\]
is a non-empty open set of  $\R^n$.
Therefore the dehomogenized series $\phi_\tau(\lambda;x)$ converge in an open set
in the $(s,x)$ space.

\begin{theorem}  \label{th:converge}
The dehomogenized series $\phi_\tau(\lambda; x)$ satisfies the hypergeometric
differential-difference system $\HH_A$ and they are linearly independent
convergent solutions of $\HH_A$ when $\lambda$ runs over admissible
exponents associated to the initial system induced 
by the weight vector $\tilde{w}(\varepsilon)$.
\end{theorem}

\begin{proof}
Since $A \lambda = s$, it is easy to show that they are formal solutions of
the differential-difference system $\HH_A$.
We will prove that we can construct $m$  linearly independent solutions.
We note that the weight vector
$\tilde{w}(\varepsilon) = (1, \ldots, 1, 0) + \varepsilon v \in
\R^{n+1}$ is in the neighborhood of $(1,\ldots,1,0)\in \R^{n+1}$ 
and in the interior of a maximal dimensional Gr\"obner cone of 
$I_{\tilde{A}}$.

It follows from \cite[p.119]{SST}  that
the minimal generating set of  
$\init_{(1,\ldots,1,0)}I_{\tilde{A}}$ does not contain  $\pd{n+1}$.
Since
\[
\init_{\tilde{w}(\varepsilon)} I_{\tilde{A}} = \init_v(\init_{(1,\ldots,1,0)}I_{\tilde{A}})
\]
does not contain $\pd{n+1}$, we have
\[
M
=
\langle \init_{\tilde{w}(\varepsilon)} I_{\tilde{A}} \rangle
=
\langle \init_{w(\varepsilon)} I_{A} \rangle
\qquad \mbox{in $\C[\pd{1},\ldots,\pd{n+1}]$}.
\]
Here, we define  $w(\varepsilon)$
with $\tilde{w}(\varepsilon) = (w(\varepsilon),0)$.
Put
$\tilde{\theta} = (\theta_1, \ldots, \theta_{n+1})$.
From \cite[Theorem 3.1.3]{SST}, for 
generic $\tilde{\beta} = (\beta_0, \beta),\ \  \beta \in \C^d$,
the initial ideal 
$\init_{(-\tilde{w}(\varepsilon),\tilde{w}(\varepsilon))} H_{\tilde{A}}(\tilde{\beta})$
is generated by
$\init_{\tilde{w}(\varepsilon)} (I_{\tilde{A}})$
and
$\tilde{A}\tilde{\theta}- \tilde{\beta}$.
Let us denote by  $T(M)$ the standard pairs of $M$.
From \cite[Theorem 3.2.10]{SST},
the initial ideal
\begin{equation}\label{eqn:sys:5.2}
\langle\init_{\tilde{w}(\varepsilon)}I_{\tilde{A}},
\tilde{A}\tilde{\theta}- \tilde{\beta}\rangle
\end{equation}
has  $\#T(M)=\vol(\tilde{A})$ linearly independent solutions of the form
\[
\{ \tilde{x}^{\tilde{\lambda}} \mid (\partial^a, T) \in T(M) \}
\]
Here, $\tilde{\lambda}$ is defined by
$
\tilde{\lambda}_i = a_i \in \Z_{\ge 0}, \ \forall i\not\in T
$
and 
$\tilde{A}\tilde{\lambda}=\tilde{\beta}$.
Note that $\tilde{\lambda}$ is admissible for the $d$-simplex $T$.

Since we have
\[
\langle \init_{\tilde{w}(\varepsilon)} I_{\tilde{A}} \rangle
=
\langle \init_{w(\varepsilon)} I_{A} \rangle
\]
the difference between
\begin{equation}\label{eqn:sys:5.1}
\langle\init_{{w}(\varepsilon)}I_{{A}}, {A}{\theta}- {\beta}\rangle
\end{equation}
and $(\ref{eqn:sys:5.2})$ is only
\[
\theta_1 + \cdots + \theta_n + \theta_{n+1} - \beta_0
\]
and other equations do not contain 
$x_{n+1}, \pd{n+1}$.

For any $(\partial^a, T)\in T(M)$, we have $n+1\in T$.
Therefore, the two solution spaces
$(\ref{eqn:sys:5.1})$ and $(\ref{eqn:sys:5.2})$ 
are isomorphic under the correspondence
\begin{equation}\label{eq:isom:5.1}
x^\lambda \mapsto \tilde{x}^{\tilde{\lambda}}
\end{equation}
Here, we put $\tilde{\lambda} = (\lambda, \lambda_{n+1})$
and
$\lambda_{n+1}$ is defined by
\[
\sum_{i=1}^n\lambda_i + \lambda_{n+1} - \beta_0 = 0
\]
It follows from \cite[Theorem 2.3.11 and Theorem 3.2.10]{SST} that
\[
\{
\tilde{x}^{\tilde{\lambda}} \mid (\partial^a, T) \in T(M)\}
\]
are $\C$-linearly independent.
Therefore, 
from the correspondence $(\ref{eq:isom:5.1})$, 
the functions
\[
\{ x^\lambda \mid (\partial^a, T) \in T(M)\},
\]
of which cardinality is $\vol(A)$, 
are $\C$-linearly independent.
Hence, series solutions with the initial terms
\[
\left\{ \frac{x^\lambda}{\Gamma(\lambda+1)} \mid (\partial^a, T) \in T(M)
\right\}
\]
are $\C$ linearly independent,
which implies the linear independence of series solutions
with these starting terms \cite{SST}. 
We have completed the proof of the theorem
and also that of Theorem \ref{thm:solutions}.
\end{proof}

\begin{theorem}\label{th:main}
The holonomic rank of $\HH_A$ is equal to the normalized volume of $A$.
\end{theorem}

\begin{proof}
First we will prove $\rank(\HH_A) \le \vol(A)$.
It follows from the Adolphson's theorem~(\cite{Adolphson}) that
the holonomic rank of $\A$-hypergeometric system
$H_A(\beta)$ is equal to the normalized volume of $A$ for generic
parameters $\beta$.
It implies that the standard monomials for a Gr\"obner basis of
the $\A$-hypergeometric system $H_A(s)$
in $\C(s,x)\langle \pd{1}, \ldots, \pd{n} \rangle$
consists of $\vol(A)$ elements.
We note that elements in the Gr\"obner basis can be regarded as an element in
the ring of differential-difference operators with rational function coefficients
$\U$.
We denote by $\pd{j}$ and $r_j$ the creation and annihilation operators.
The existence of them are proved in 
\cite[Chapter 4]{SST-compositio}.
Then, we have 
\[
H_j =  \pd{j} - \prod_{i=1}^n S_{i}^{-a_{ij}} \in \HH_A
\]
and
\[
B_j = r_j     - \prod_{i=1}^n S_{i}^{a_{ij}} \in \HH_A, \quad
r_j \in \C(s,x)\langle \pd{1}, \ldots, \pd{n} \rangle.
\]
Since the column vectors of $A$ generate the lattice $\Z^d$,
we obtain from $B_j$'s and $H_j$'s elements of the form
$S_{i} - p(s,x,\pd{}), \ S_{i}^{-1} - q(s,x,\pd{}) \in \HH_A$.
It implies the number of standard monomials
of a Gr\"obner basis of $\HH_A$
with respect to a block order such that
$S_1, \ldots, S_n > S_1^{-1}, \ldots, S_n^{-1} > \pd{1}, \ldots, \pd{n}$ 
is less than or equal to
$\vol(A)$.


Second, we will prove $\rank(\HH_A) \ge \vol(A)$.
We suppose that $\rank(\HH_A) < \vol(A)$ and will induce a contradiction.
For the block order $S_1, \cdots, S_d > S_1^{-1}, \cdots, S_d^{-1} > \pd{1},
\cdots, \pd{n}$, we can show that the standard monomials $T$ of a Gr\"obner
basis of $\HH_A$ in $\U$ contains only differential terms and
$\# T < \vol (A)$ by the assumption.
Let $T'$ be the standard monomials of Gr\"obner basis $G(s)$ of $H_A(s)$
in the ring of differential operators with rational function coefficients $D(s)$.
Note that $\# T' = {\rm vol}(A)$.
Then $T$ is a proper subset of the set $T'$.
For $r \in T'\setminus T$, it follows that
\[
\partial^r \equiv \sum_{\alpha\in T} c_\alpha(x,s) \partial^\alpha \qquad
\mod \HH_A.
\]
From Theorem~\ref{th:converge}, we have
convergent series solutions $f_1(s,x), \cdots, f_m(s,x)$ of $\HH_A$, where
$m=\vol(A)$.
So,
\begin{equation}\label{eq:lindep}
\partial^r \bullet f_i = \sum_{\alpha\in T} c_\alpha(x,s)
\partial^\alpha \bullet f_i
\end{equation}
Since $f_1(s,x), \ldots, f_m(s,x)$ are linearly independent, the
Wronskian standing for $T'$
\[
W(T';f)(x,s) =
\left|
\begin{matrix}
f_1(s;x)             & \cdots & f_m(\beta;x) \cr
\pd{}^\delta f_1(s;x)& \cdots & \pd{}^\delta f_m(\beta;x) \cr
\vdots               & \cdots & \vdots \cr
\end{matrix}
\right|
\qquad (\pd{}^\delta\in T' \setminus \{1\})
\]
is non-zero for generic number $s$.
However $r \in T'$ and $(\ref{eq:lindep})$ induce the Wronskian
$W(T';f)(s,x)$ is equal to zero.

Finally, by $\rank(\HH_A) \le \vol(A)$ and $\rank(\HH_A) \ge \vol(A)$,
the theorem is proved.
\end{proof}

\begin{example} \label{ex:series} \rm
Put
$\displaystyle
A =
\begin{pmatrix}
1 & 2 & 3
\end{pmatrix}
$ and
$\displaystyle
\tilde{A} =
\begin{pmatrix}
1 & 1 & 1 & 1 \\
1 & 2 & 3 & 0
\end{pmatrix}
$.
This is {\it Airy type integral} \cite[p.223]{SST}.

The matrix $\tilde{A}$ is homogeneous.
For ${\tilde w}(\varepsilon) = (1,1,1,0)+\frac{1}{100}(1,0,0,0)$, 
the initial ideal 
${\rm in}_{{\tilde w}(\varepsilon)}(I_{\tilde A})$ is generated
by $\partial_1^2, \partial_1 \partial_2, \partial_1 \partial_3, \partial_2^3$.
Note that the initial ideal does not contain $\partial_4$.
We solve the initial system
$\left( \tilde{A} \tilde{\theta} - \tilde{s} \right) \bullet g = 0, 
\left( {\rm in}_{{\tilde w}(\varepsilon)}(I_{\tilde A}) \right) \bullet g = 0$.
The standard pairs $(\partial^a, T)$ for 
${\rm in}_{{\tilde w}(\varepsilon)}(I_{\tilde A})$
are 
$(\partial_1^0 \partial_2^1, \{ 3, 4 \})$,
$(\partial_1^0 \partial_2^0, \{ 3, 4 \})$,
$(\partial_1^0 \partial_2^2, \{ 3, 4 \})$.
Hence, the solutions for the initial system are \\
$x_1^0 x_2^1 x_3^{(s_1-2)/3} x_4^{s_0-1-(s_1-2)/3}$,
$x_1^0 x_2^0 x_3^{s_1/3} x_4^{a_0-s_1/3}$,
$x_1^0 x_2^2 x_3^{(s_1-4)/3} x_4^{s_0-2-(s_1-4)/3}$
(\cite{SST}).
Therefore, the ${\cal A}$-hypergeometric differential-difference system
$\HH_{\tilde{A}}$ has the following series solutions.

\begin{eqnarray*}
\tilde{\phi}_1(\tilde{\lambda}, \tilde{x}) 
&=&
x_4^{s_0}
\left(\frac{x_2}{x_4}\right)
\left(\frac{x_3}{x_4}\right)^{\frac{s_1-2}{3}} \\
& & \ \cdot 
\sum_{\substack{
k_1 \ge 0,\ k_2\ge -1\\
(k_1,k_2) \in L'
}}
\frac{
\left(
{x_1}{x_3^{-1/3}x_4^{-2/3}}
\right)^{k_1}
\left(
{x_2}{x_3^{-2/3}x_4^{-1/3}}
\right)^{k_2}
}
{
k_1!(k_2+1)!
\Gamma(\frac{s_1-k_1-2k_2+1}{3})
\Gamma(\frac{3s_0-s_1-2k_1-k_2+2}{3})
}
\\
\tilde{\phi}_2(\tilde{\lambda}, \tilde{x}) 
&=&
x_4^{s_0}
\left(\frac{x_3}{x_4}\right)^{\frac{s_1}{3}} \\
& & \ \cdot
\sum_{\substack{
k_1 \ge 0,\ k_2\ge 0\\
(k_1,k_2) \in L'
}}
\frac{
\left(
{x_1}{x_3^{-1/3}x_4^{-2/3}}
\right)^{k_1}
\left(
{x_2}{x_3^{-2/3}x_4^{-1/3}}
\right)^{k_2}
}
{
k_1!k_2!
\Gamma(\frac{s_1-k_1-2k_2+3}{3})
\Gamma(\frac{3s_0-s_1-2k_1-k_2+3}{3})
}
\\
\tilde{\phi}_3(\tilde{\lambda}, \tilde{x}) 
&=&
x_4^{s_0}
\left(\frac{x_2}{x_4}\right)^{2}
\left(\frac{x_3}{x_4}\right)^{\frac{s_1-4}{3}} \\
& & \cdot
\sum_{\substack{
k_1 \ge 0,\ k_2\ge -2\\
(k_1,k_2) \in L'
}}
\frac{
\left(
{x_1}{x_3^{-1/3}x_4^{-2/3}}
\right)^{k_1}
\left(
{x_2}{x_3^{-2/3}x_4^{-1/3}}
\right)^{k_2}
}
{
k_1!(k_2+2)!
\Gamma(\frac{s_1-k_1-2k_2-1}{3})
\Gamma(\frac{3s_0-s_1-2k_1-k_2+1}{3})
}
\end{eqnarray*}
Here, 
\[L' = \{ (k_1, k_2) \,|\, k_1 \equiv 0 \ {\rm mod}\, 3, 
                                k_2 \equiv 0 \ {\rm mod}\, 3 \} \cup 
       \{ (k_1, k_2) \,|\, k_1 \equiv 1 \ {\rm mod}\, 3, 
                                k_2 \equiv 1 \ {\rm mod}\, 3 \}.
\]
The matrix $A$ is not homogeneous and by dehomogenizing 
the series solution for $\tilde{A}$
we obtain the following series solutions
for the ${\cal A}$-hypergeometric differential-difference
system $\HH_A$.
\begin{eqnarray*}
{\phi}_1({\lambda}, {x}) 
&=&
x_2
x_3^{\frac{s_1-2}{3}}
\sum_{\substack{
k_1 \ge 0,\ k_2\ge -1\\
(k_1,k_2) \in L'
}}
\frac{
\left(
{x_1}{x_3^{-1/3}}
\right)^{k_1}
\left(
{x_2}{x_3^{-2/3}}
\right)^{k_2}
}
{
k_1!(k_2+1)!
\Gamma(\frac{s_1-k_1-2k_2+1}{3})
}
\\
{\phi}_2({\lambda}, {x}) 
&=&
x_3^{\frac{s_1}{3}}
\sum_{\substack{
k_1 \ge 0,\ k_2\ge 0\\
(k_1,k_2) \in L'
}}
\frac{
\left(
{x_1}{x_3^{-1/3}}
\right)^{k_1}
\left(
{x_2}{x_3^{-2/3}}
\right)^{k_2}
}
{
k_1!k_2!
\Gamma(\frac{s_1-k_1-2k_2+3}{3})
}
\\
{\phi}_3({\lambda}, {x}) 
&=&
x_2^2
x_3^{\frac{s_1-4}{3}}
\sum_{\substack{
k_1 \ge 0,\ k_2\ge -2\\
(k_1,k_2) \in L'
}}
\frac{
\left(
{x_1}{x_3^{-1/3}}
\right)^{k_1}
\left(
{x_2}{x_3^{-2/3}}
\right)^{k_2}
}
{
k_1!(k_2+2)!
\Gamma(\frac{s_1-k_1-2k_2-1}{3})
}
\end{eqnarray*}
Here $\phi_k(x)$ is the dehomogenization of $\tilde{\phi}_k(x)$.

Finally, let us present a difference Pfaffian system
for $A$.
It can be derived by using Gr\"obner bases of $\HH_A$
and has the following form:
\[
S_1 
\left(
\begin{array}{ccc}
f \\ x_3 \partial_3 \bullet f \\ S_1 \bullet f
\end{array}
\right)
=
\left(
\begin{array}{ccc}
0 & 0 & 1 \\
- \frac{s_1 x_1}{6x_2} & \frac{3 x_1x_3- 4 x_2^2}{6 x_2x_3} & 
\frac{2(s_1- 1)x_2+ x_1^2}{6x_2} \\
\frac{s_1}{2x_2} & -\frac{3}{2x_2} & -\frac{x_1}{2x_2}
\end{array}
\right)
\left(
\begin{array}{ccc}
f \\ x_3 \partial_3 \bullet f \\ S_1 \bullet f
\end{array}
\right).
\]
\end{example}

\end{document}